\documentclass[12pt]{amsart}

\usepackage{amsrefs}
\usepackage{enumerate}
\usepackage{hyperref}
\hypersetup{hidelinks,colorlinks}
\usepackage{amsmath}
\usepackage{amssymb}
\usepackage{mathtools}
\usepackage{mathrsfs}
\usepackage{color}

\usepackage[normalem]{ulem}
\usepackage{comment}
\usepackage{nicefrac}

\usepackage{booktabs, floatrow, makecell}
\hypersetup{
    linkcolor=red,
    citecolor=blue,
 }

\def\Xint#1{\mathchoice
{\XXint\displaystyle\textstyle{#1}}%
{\XXint\textstyle\scriptstyle{#1}}%
{\XXint\scriptstyle\scriptscriptstyle{#1}}%
{\XXint\scriptscriptstyle\scriptscriptstyle{#1}}%
\!\int}
\def\XXint#1#2#3{{\setbox0=\hbox{$#1{#2#3}{\int}$ }
\vcenter{\hbox{$#2#3$ }}\kern-.6\wd0}}

\def\dashint{\Xint-}

\usepackage{xcolor}

\newtheorem{theorem}{Theorem}[section]
\newtheorem{prop}[theorem]{Proposition}
\newtheorem{lemma}[theorem]{Lemma}

\newtheorem{definition}[theorem]{Definition}

\newtheorem{openquestion[theorem]}{Question}
\newtheorem{conj}[theorem]{Conjecture}

\newcommand{\RR}{\mathbb{R}}
\newcommand{\NN}{\mathbb{N}}

\newcommand{\ZZ}{\mathbb{Z}}

\newcommand{\Ei}{\text{\normalfont Ei}}
\newcommand{\li}{\text{\normalfont li}}

\newsavebox\foobox
\newlength{\foodim}

\newcommand{\defeq}{\vcentcolon=}

\usepackage{graphicx}

\title[Asymptotics of the reciprocal Hardy-Littlewood integrals]{Asymptotic expansions for the reciprocal Hardy-Littlewood logarithmic integrals}
\author{Glenn Bruda}
\thanks{\textit{Email:} \texttt{\href{mailto:glenn.bruda@ufl.edu}{glenn.bruda@ufl.edu}}}
\date{\today}

\allowdisplaybreaks

\begin{document}

\begin{abstract}
      Defining a family of recurrences, we generalize Comtet's formula for the generating function of the enumeration of indecomposable permutations. Consequently, we generalize Panaitopol's asymptotic expansion for the prime counting function, obtaining asymptotic expansions salient to the first Hardy-Littlewood conjecture.
\end{abstract}

\maketitle

\section{Introduction}

Preliminarily, a \emph{permutation} $\sigma$ of length $n\in\ZZ^+$ is a bijection $\sigma:\{1,2,\dots,n\}\to\{1,2,\dots,n\}$. Given a permutation $\sigma$ of length $n$, the one-line notation of $\sigma$ is $m_1 m_2 \cdots m_n$, where each $m_i=\sigma(i)$. For example, $\sigma_1=12345$ and $\sigma_2=23451$ are permutations of length $5$, where $\sigma_1$ is the identity mapping and $\sigma_2(k)=k+1\pmod{5}$ for all $k\in\{1,2,3,4,5\}$. A \emph{proper prefix} of a permutation $\sigma$ of length $n$ is a restriction $\sigma|_{\{1,2,\dots,k\}}$ for some $1\leq k<n$. For example, the proper prefixes of $23451$ are $2345,234,23,2$.

As seen in \cite{indecomp_perms}, an \emph{indecomposable permutation}\footnote{Also known as a \emph{connected permutation} or an \emph{irreducible permutation} according to \cite{OEIS_indecomp}.} is a permutation such that no proper prefix of itself is a permutation. For example, $321,312,231$ are all indecomposable since $2,3,31,32,23$ are not permutations. On the other hand, $123,132,213$ are not indecomposable since $1,12,21$ are all permutations. A thorough introduction to indecomposable permutations also appears in \cite{indecomp_perms} and the OEIS sequence of their enumeration is seen at \cite{OEIS_indecomp}. 

All power series in this paper should be considered as {formal power series} over the integers. A \emph{formal power series} over a ring $R$ is an infinite series of monomials
\begin{align*}
    c_0+c_1x+c_2x^2+c_3x^3+\cdots
\end{align*}
such that each coefficient $c_i\in R$. An introduction to formal power series may be found in \cite{formal_power_series}.

Comtet introduced indecomposable permutations in \cite{comtet}, motivated by determining the coefficients of the reciprocal of $\sum_{n\geq0}n!x^n$. In particular, he found that
\begin{align*}
    1-\left(\sum_{n\geq0}n!x^n\right)^{-1} \ = \ \sum_{n\geq1}\mathcal{I}_n x^n,
\end{align*}
where $\mathcal{I}_n$ is the number of indecomposable permutations of length $n$. We generalize this result in Proposition \ref{comtet_generalization}, determining the coefficients of the reciprocal of $\sum_{n\geq0}(n+k-1)!x^n$ for any $k\in\ZZ^+$.

We use Proposition \ref{comtet_generalization} to determine the asymptotic behavior of the reciprocals of a class of integrals appearing in Conjecture \ref{hardy_littlewood_conj}. In particular, we obtain {asymptotic expansions} for said reciprocal functions.

\begin{definition}[Asymptotic expansion, {\cite[Section 2.2]{def_of_asymptotic_expansion}}]
    Let $\{\varphi_n\}_{n\in\NN}$ be a sequence of real-valued functions such that $\varphi_{n+1}(x)=o(\varphi_{n}(x))$ as $x\to\infty$ for every $n\in\NN\defeq\{0,1,2,\dots\}$. Let $f$ be a real-valued function. Suppose for every $N\in\NN$ that
    \begin{align*}
        f(x)-\sum_{n=0}^{N}a_n\varphi_n(x) \ = \ o(\varphi_N(x))
    \end{align*}
    as $x\to\infty$, where $\{a_n\}$ is a sequence of real numbers. Then we say that $\sum_{n\geq0}a_n\varphi_n(x)$ is the asymptotic expansion of $f(x)$ with respect to the asymptotic sequence $\varphi_n(x)$, writing
    \begin{align*}
        f(x) \ \sim \ \sum_{n\geq0}a_n\varphi_n(x).
    \end{align*}
\end{definition}

In \cite{panaitopol}, Panaitopol showed that the enumeration of indecomposable permutations appeared in an asymptotic expansion of the prime counting function. An extensive discussion on asymptotic expansions of the prime counting function may be found in \cite{asymptotic_expansions_of_prime_counting_func}.

\begin{prop}[Panaitopol, {\cite{panaitopol}}] \label{panaitopol}
    Letting $\pi(x)$ denote the prime counting function, we have the asymptotic expansion
    \begin{align*}
        \frac{1}{\pi(x)} \ \sim \ \frac{\log x}{x}\left(1-\frac{1}{\log x}-\sum_{n\geq2}\frac{\mathcal{I}_n}{(\log x)^{n}}\right),
    \end{align*}
    where $\mathcal{I}_n$ denotes the number of indecomposable permutations of length $n$.
\end{prop}

We remark that Proposition \ref{panaitopol} is equivalent to the $k=1$ case of Theorem \ref{main_result} since $\li(x)^{-1}$ and $\pi(x)^{-1}$ have the same asymptotic expansion with respect to the asymptotic sequence $x^{-1}(\log x)^{-n+1}$. Indeed, letting $n\in\NN$ be arbitrary and appealing to the classical estimate
\begin{align*}
    \pi(x) \ = \ \li(x)+O\left(xe^{-a\sqrt{\log x}}\right) \text{~for some~}a>0
\end{align*}
of de la Vall\'ee Poussin \cite{poussin}, we have
\begin{align*}
    &\phantom{=}\frac{x}{\pi(x)\log^{-n+1}(x)}-\frac{x}{\li(x)\log^{-n+1}(x)} \ = \ x\log^{n-1}(x)\left(\frac{1}{\pi(x)}-\frac{1}{\li(x)}\right)\\
    &= \ x\log^{n-1}(x)\left(\frac{1}{\li(x)+O(xe^{-a\sqrt{\log x}})}-\frac{1}{\li(x)}\right) \ = \ x\log^{n-1}(x)\left(\frac{O(xe^{-a\sqrt{\log x}})}{\li(x)(\li(x)+O(xe^{-a\sqrt{\log x}}))}\right)\\
    &= \ x\log^{n-1}(x)~O\left(\frac{xe^{-a\sqrt{\log x}}}{\li(x)^2}\right) \ = \ x\log^{n-1}(x)~O\left(\frac{xe^{-a\sqrt{\log x}}}{x^2/\log^2(x)}\right) \ = \ O\left(e^{-a\sqrt{\log x}}\log^{n+1}(x)\right),
\end{align*}
which converges to zero for any $a>0$ since the last function is globally nonnegative and decreasing for $x>\exp((2n+2)^2/a^2)$. With this remark noted, Proposition \ref{panaitopol} may be quickly obtained from Comtet's generating function formula by following the $k=1$ case of the proof of Theorem \ref{main_result}.

The famous prime number theorem, equivalently asserting that the prime counting function is asymptotic to the logarithmic integral, is generalized by the first Hardy-Littlewood conjecture.


\begin{conj}[First Hardy-Littlewood Conjecture, {\cites{discussion_of_HLC,discussion_of_HLC_2,hardy_littlewood}}]\label{hardy_littlewood_conj}
    Let $0<m_1<m_2<\cdots<m_k$ be integers. Unless $P=(p,p+2m_1,p+2m_2,\dots,p+2m_k)$ forms a complete residue class with respect to some prime, the number of primes $p\leq x$ such that $p,p+2m_1,p+2m_2,\dots,p+2m_k$ are prime, denoted $\pi_P(x)$, is asymptotic\footnote{We say that functions $f,g$ are \emph{asymptotic} if $\lim_{x\to\infty}f(x)/g(x)=1$.} to
    \begin{align*}
        2^k\prod_{\substack{q\text{~\normalfont prime}\\ q\geq3}}\left(1-\frac{1}{q}\right)^{-k-1}\left(1-\frac{w(q;2m_1,2m_2,\dots,2m_k)}{q}\right)\int_{2}^{x}\frac{1}{(\log t)^{k+1}}dt,
    \end{align*}
    where $w(q;2m_1,2m_2,\dots,2m_k)$ is the number of distinct residues of $0,2m_1,2m_2,\dots,2m_k$ modulo $q$.
\end{conj}
Like the prime number theorem, Conjecture \ref{hardy_littlewood_conj} does not concern an asymptotic expansion; rather, only the lowest order asymptotic term of $\pi_P(x)$ is regarded. Theorem \ref{main_result}, the main result of this paper, gives the asymptotic expansion
    \begin{align*}
        \left(\dashint_{0}^{x}\frac{1}{(\log t)^k}dt\right)^{-1} \ \sim \ \frac{(\log x)^k}{x}\left(1-\frac{k}{\log x}-\sum_{n\geq1}\frac{a_{n}^{(k)}}{(\log x)^{n+1}}\right)
    \end{align*}
for any $k\in\ZZ^+$, where the dashed integral denotes the Cauchy principal value and the sequence $a_n^{(k)}$ is as defined in Definition \ref{sequence_def}. In particular, Theorem \ref{main_result} provides auxiliary information for a potential strengthening of Conjecture \ref{hardy_littlewood_conj} such that we may eventually obtain (or at least conjecture) a complete asymptotic expansion for $\pi_P(x)$.

\section{Results}

To prove Theorem \ref{main_result}, we first generalize Comtet's generating function formula in Proposition \ref{comtet_generalization}. To this end, we define the following family of recurrences.

\begin{definition}\label{sequence_def}
    For $k\in\ZZ^+$, let $a_n^{(k)}$ be the sequence defined by the recurrence
    \begin{align*}
        (k-1)!a_{n}^{(k)} \ = \ (n+k)!-k(n+k-1)!-\sum_{m=0}^{n-2}(n-m+k-2)!a_{m+1}^{(k)}
    \end{align*}
    with $a_0^{(k)}=1$.
\end{definition} 
As noted in \cite[pg.357]{asymptotic_expansions_of_prime_counting_func}, $a_n^{(1)}$ is the number of indecomposable permutations of length $n+1$. The sequences for $k=2,3,4$ and their respective OEIS A-numbers are listed in Table \ref{1L1}.

\begin{table}[h]\label{table1}
  \small\centering
  \begin{floatrow}
    \ttabbox{ \caption{The first few terms of $a_n^{(2)},a_n^{(3)},a_n^{(4)}$ and their OEIS A-numbers.} \label{1L1}}
    {
    \begin{tabular}{|c c c c c c c c c c c|} 
 \hline
 $n$ & 0 & 1 & 2 & 3 & 4 & 5 & 6 & 7 & 8 & 9 \\
 \hline
$a_n^{(2)}$ (\href{https://oeis.org/A111537}{A111537}) & 1 & 2 & 8 & 44 & 296 & 2312 & 20384 & 199376 & 2138336 & 24936416 \\ 
 
$a_n^{(3)}$ (\href{https://oeis.org/A111546}{A111546}) & 1 & 3 & 15 & 99 & 783 & 7083 & 71415 & 789939 & 9485343 & 122721723 \\
 
$a_n^{(4)}$ (\href{https://oeis.org/A111556}{A111556}) & 1 & 4 & 24 & 184 & 1664 & 17024 & 192384 & 2366144 & 31362304 & 444907264 \\
 \hline
\end{tabular}
    }
  \end{floatrow}
\end{table}

\begin{prop}\label{comtet_generalization}
    Let $k\in\ZZ^+$. Then
    \begin{align*}
        1-(k-1)!\left(\sum_{n\geq 0}{(n+k-1)!} x^{n}\right)^{-1} \ = \ kx+\sum_{n\geq1}a_{n}^{(k)}x^{n+1},
    \end{align*}
    taken as formal power series.
\end{prop}

\begin{proof}
    Let $k\in\ZZ^+$. Let $I_k(x)\defeq kx+\sum_{n\geq1}a_{n}^{(k)}x^{n+1}$ and for $m\in\NN$, let $A_m(x)\defeq\sum_{n\geq0}(n+m)!x^n$. Then by the definition of $a_n^{(k)}$,
    \begin{align*}
        &(k-1)!I_k(x) \ = \ k!x-\sum_{n\geq1}(k-1)!a_{n}^{(k)}x^{n+1}\nonumber\\
        &= \ k!x+x(A_{k}(x)-k!)-kx (A_{k-1}(x)-(k-1)!)-x\sum_{n\geq1}\sum_{m=0}^{n-2}(n-m+k-2)!a_{m+1}^{(k)}x^n\nonumber\\
        &= \ k!x+xA_{k}(x)-kx A_{k-1}(x)-x^3\sum_{n\geq0}\sum_{m=0}^{n}(n-m+k)!a_{m+1}^{(k)}x^{n}.
    \end{align*}
    Noting that
    \begin{align*}
        \sum_{n\geq0}a_{n+1}^{(k)}x^n \ = \ x^{-2}\sum_{n\geq1}a_n^{(k)}x^{n+1} \ = \ x^{-2}(I_k(x)-kx),
    \end{align*}
    we obtain
    \begin{align*}
        (k-1)!I_k(x) \ &= \ k!x+xA_{k}(x)-kx A_{k-1}(x)-x(I_k(x)-kx)A_k(x)\\
        &= \ k!x-kx A_{k-1}(x)-x(I_k(x)-kx-1)A_k(x).
    \end{align*}
    It then follows that
    \begin{align*}
        I_k(x) \ = \ \frac{k!x+(1+kx)xA_k(x)-kxA_{k-1}(x)}{(k-1)!+x A_k(x)}.
    \end{align*}
    We may now simplify the above since 
    \begin{align*}
        (k-1)!+xA_k(x) \ = \ A_{k-1}(x).
    \end{align*}
    Indeed, we thusly have
    \begin{align*}
        I_k(x) \ &= \ \frac{k!x+(1+kx)\left(A_{k-1}(x)-(k-1)!\right)-kx A_{k-1}(x)}{A_{k-1}(x)}\\
        &= \ 1-\frac{(k-1)!}{A_{k-1}(x)}\\
        &= \ 1-(k-1)!\left(\sum_{n\geq 0}{(n+k-1)!} x^{n}\right)^{-1},
    \end{align*}
    as desired.
\end{proof}

Now with Proposition \ref{comtet_generalization} proven, we define the \emph{exponential integral} and the \emph{Hardy-Littlewood logarithmic integrals}. Then, after invoking the standard expansion stated in Proposition \ref{exp_int_asymptotic_expansion} and inductively proving Lemma \ref{formula_for_li_k}, we finally prove Theorem \ref{main_result}.

\begin{definition}[Exponential integral]
    We define $\Ei:\RR\backslash\{0\}\to\RR$ by
    \begin{align*}
        \Ei(x) \ \defeq \ \dashint_{-\infty}^{x}\frac{e^t}{t}dt,
    \end{align*}
    where the dashed integral denotes the Cauchy principal value.
\end{definition}

\begin{prop}[{\cite[\href{https://dlmf.nist.gov/6.12\#E2}{Equation 6.12.2}]{NIST:DLMF}}]\label{exp_int_asymptotic_expansion}
    As $x\to\infty$, we have the asymptotic expansion
    \begin{align*}
        \Ei(x) \ \sim \ \frac{e^x}{x}\sum_{n\geq0}\frac{n!}{x^n}.
    \end{align*}
\end{prop}

\begin{definition}[Hardy-Littlewood logarithmic integrals]
    For $k\in\ZZ^+$, we define $\text{\normalfont li}_k:\RR\backslash\{1\}\to\RR$ by
    \begin{align*}
        \text{\normalfont li}_k(x) \ \defeq \ \dashint_{0}^{x}\frac{1}{(\log t)^k}dt,
    \end{align*}
    where the dashed integral denotes the Cauchy principal value. Note that $\text{\normalfont li}_1(x)$ is the usual logarithmic integral, denoted $\text{\normalfont li}(x)$.
\end{definition}

\begin{lemma}\label{formula_for_li_k}
    Let $k\in\ZZ^+$. Then
    \begin{align*}
        \text{\normalfont li}_k(x) \ = \ \frac{1}{(k-1)!}\left(\text{\normalfont li}(x)-x\sum_{i=0}^{k-2}\frac{i!}{(\log x)^{i+1}}\right).
    \end{align*}
\end{lemma}

\begin{proof}
    This result lends a simple inductive proof via integration by parts. For $k=1$, the claim follows from the definition of the logarithmic integral. For the inductive step, note that
    \begin{align*}
        \text{\normalfont li}_k(x) \ = \ \frac{x}{(\log x)^k}+\dashint_{0}^{x}\frac{k}{(\log t)^{k+1}}dt.
    \end{align*}
    So by inductive hypothesis,
    \begin{align*}
        \text{\normalfont li}_{k+1}(x) \ = \ \dashint_{0}^{x}\frac{1}{(\log t)^{k+1}}dt \ &= \ \frac{1}{k!}\left(\text{\normalfont li}(x)-x\sum_{i=0}^{k-2}\frac{i!}{(\log x)^{i+1}}\right)-\frac{x}{k(\log x)^k}\nonumber \\
        &= \ \frac{1}{k!}\left(\text{\normalfont li}(x)-x\sum_{i=0}^{k-1}\frac{i!}{(\log x)^{i+1}}\right).
    \end{align*}
\end{proof}

\begin{theorem}\label{main_result}
    Let $k\in\ZZ^+$. Then we have the asymptotic expansion
    \begin{align*}
        \frac{1}{\text{\normalfont li}_k(x)} \ \sim \  \frac{(\log x)^k}{x}\left(1-\frac{k}{\log x}-\sum_{n\geq1}\frac{a_{n}^{(k)}}{(\log x)^{n+1}}\right).
    \end{align*}
\end{theorem}

\begin{proof}
    Let $k\in\ZZ^+$ and $A(x)\defeq\sum_{n\geq0}n!x^n$. By Proposition \ref{comtet_generalization},
    \begin{align*}
    kx+\sum_{n\geq1}a_{n}^{(k)}x^{n+1} \ &= \ 1-(k-1)!\left(\sum_{n\geq 0}{(n+k-1)!} x^{n}\right)^{-1}\nonumber \\ \nonumber
    &= \ 1-(k-1)!\left(\sum_{n\geq k-1}{n!} x^{n-k+1}\right)^{-1}\\ \nonumber
    &= \ 1-\frac{(k-1)!x^{k-1}}{\sum_{n\geq k-1}n! x^n}\\
    &= \ 1-\frac{(k-1)!x^{k-1}}{A(x)-\sum_{i=0}^{k-2}i!x^i}.
\end{align*}
After a change of variable, we obtain
\begin{align*}
    \frac{(k-1)!x^{1-k}}{A(1/x)-\sum_{i=0}^{k-2}i!x^{-i}} \ = \ 1-\frac{k}{x}-\sum_{n\geq1}\frac{a_{n}^{(k)}}{x^{n+1}}.
\end{align*}
Proposition \ref{exp_int_asymptotic_expansion} gives the relation $A(1/x)\sim xe^{-x}\Ei(x)$; thus,
\begin{align*}
    1-\frac{k}{x}-\sum_{n\geq1}\frac{a_{n}^{(k)}}{x^{n+1}} \ &\sim \  \frac{(k-1)!x^{1-k}}{xe^{-x}\Ei(x)-\sum_{i=0}^{k-2}i!x^{-i}}\nonumber \\
    &= \ \frac{(k-1)!}{x^k e^{-x}\left(\Ei(x)-e^x\sum_{i=0}^{k-2}i!x^{-i-1}\right)}.
\end{align*}
Noting that $\Ei(x)=\text{li}(e^x)$, by Lemma \ref{formula_for_li_k} we have
\begin{align*}
    x^k e^{-x}\left(1-\frac{k}{x}-\sum_{n\geq1}\frac{a_{n}^{(k)}}{x^{n+1}}\right) \ \sim \  \frac{(k-1)!}{\Ei(x)-e^x\sum_{i=0}^{k-2}i!x^{-i-1}} \ = \ \frac{1}{\text{li}_k(e^x)}.
\end{align*}
Thus, 
\begin{align*}
    \frac{1}{\text{li}_k(x)} \ \sim \  \frac{(\log x)^k}{x}\left(1-\frac{k}{\log x}-\sum_{n\geq1}\frac{a_{n}^{(k)}}{(\log x)^{n+1}}\right),
\end{align*}
as desired.
\end{proof}

\section*{Acknowledgements}
We thank Steven J. Miller for his guidance and for recommending trying to generalize the $k=1$ case of Theorem \ref{main_result}.

\bibliography{main}{}
\bibliographystyle{plain}

\end{document}